\documentclass[12pt,reqno]{amsart}

\usepackage{amssymb}
\usepackage{amscd}
\usepackage{amsfonts}
\usepackage{setspace}
\usepackage{version}
\usepackage{hyperref}

\usepackage{graphicx}

\newtheorem{thm}{Theorem}[section]
\newtheorem*{GRRL}{Regularity Lemma}
\newtheorem*{Count}{Counting Lemma}

\newtheorem*{sumset-compliment-repeat}{Theorem \ref{sumset-compliment}}
\newtheorem*{FL}{Freiman's Lemma}
\newtheorem*{VII}{The Isoperimetric Inequality}
\newtheorem{lemma}[thm]{Lemma}
\newtheorem{prop}[thm]{Proposition}

\theoremstyle{definition}
\newtheorem{defn}[thm]{Definition}

\newtheorem{conjecture}[thm]{Conjecture}

\renewcommand{\leq}{\leqslant}
\renewcommand{\geq}{\geqslant}

\newcommand\Hom{\operatorname{Hom}}

\def\F{\mathbb{F}}
\newcommand\dist{\operatorname{dist}}
\newcommand\vol{\operatorname{vol}}
\def\R{\mathbb{R}}
\def\P{\mathbb{P}}
\def\C{\mathbb{C}}
\def\Z{\mathbb{Z}}

\def\Pr{\mathbb{P}}
\def\Q{\mathbb{Q}}
\def\N{\mathbb{N}}
\def\x{\textbf{x}}
\def\y{\textbf{y}}
\def\eps{\varepsilon}
\def\le{\leqslant}
\def\ge{\geqslant}
\def\plus{\,\hat{+}\,}
\def\dt{\,\mathrm{d} \theta}

\numberwithin{equation}{section}

\begin{document}

\title[Counting sets with small sumset]{Counting sets with small sumset and applications}


\author{Ben Green}
\address{Mathematical Institute, Andrew Wiles Building, Radcliffe Observatory Quarter, Woodstock Rd, Oxford OX2 6GG}
\email{ben.green@maths.ox.ac.uk}

\author{Robert Morris}
\address{IMPA, Estrada Dona Castorina 110, Jardim Bot\^anico, Rio de Janeiro, RJ, Brasil} 
\email{rob@impa.br}

\thanks{The authors would like to thank Louigi Addario-Berry and Luc Devroye for organising a workshop in Combinatorics and Probability at McGill's Bellairs Research Institute, Barbados. Discussions at that workshop led to the work in this paper. The first author is supported by a ERC starting grant 274938  ``Approximate algebraic structure and applications''.}

\onehalfspace
\subjclass[2000]{Primary }

\begin{abstract}
We study the number of $k$-element sets $A \subset \{1,\ldots,N\}$ with $|A + A| \leq K|A|$ for some (fixed) $K > 0$. Improving results of the first author and of Alon, Balogh, Samotij and the second author, we determine this number up to a factor of $2^{o(k)} N^{o(1)}$ for most $N$ and $k$. As a consequence of this and a further new result concerning the number of sets $A \subset \Z/N\Z$ with $|A +A| \leq c |A|^2$, we deduce that the random Cayley graph on $\Z/N\Z$ with edge density~$\frac{1}{2}$ has no clique or independent set of size greater than $\big( 2 + o(1) \big) \log_2 N$, asymptotically the same as for the Erd\H{o}s-R\'enyi random graph. This improves a result of the first author from 2003 in which a bound of $160 \log_2 N$ was obtained. As a second application, we show that if the elements of $A \subset \N$ are chosen at random, each with probability $1/2$, then the probability that $A+A$ misses exactly $k$ elements of $\N$ is equal to $\big( 2 + o(1) \big)^{-k/2}$ as $k \to \infty$. 
\end{abstract}

\maketitle

\tableofcontents

\section{Introduction}\label{sec1-intro}

One of the fundamental results in additive combinatorics is the theorem of Freiman~\cite{Frei59}, which states that every finite set of integers with \emph{bounded doubling} (that is, with $|A + A| \le K|A|$ for some fixed $K$) is contained in a generalized arithmetic progression of bounded dimension and size $\ll |A|$. Despite the importance of this theorem, very little attention has been paid to the closely-related question of the \emph{typical} structure of such a set. Motivated by the applications of this problem in~\cite{ABMS} and~\cite{G05}, we determine the number of $k$-subsets of $\{1,\ldots,N\}$ with $|A+A| \le K |A|$ up to a factor of $2^{o(k)} N^{o(1)}$ for most values of ~$N$ and~$k$. This result significantly improves bounds obtained in~\cite{ABMS,G05}, and confirms (a special case of) a conjecture of Alon, Balogh, Samotij and the second author~\cite{ABMS}. As an application, we improve a result of the first author~\cite{G05}, by showing that a random Cayley graph is essentially as good a Ramsey graph as the Erd\H{o}s-R\'enyi random graph $G(n,\frac{1}{2})$. 

Given an abelian group $\Gamma$ and a set $A \subset \Gamma$, the sumset $A+A$ and restricted sumset $A \plus A$ of $A$ are defined as follows:
$$A + A \, := \, \big\{ a + b : a,b \in A \big\} \qquad \textup{and} \qquad A \plus A \, := \, \big\{ a + b : a,b \in A, \, a \neq b \big\}.$$
Freiman's theorem (and subsequent quantitative improvements of it) gives, in some sense, a complete description of sets $A \subset \Z$ with $|A+A| \le K|A|$. Any such set is a subset of a generalised progression
\[ P = \{\ell_1 x_1 + \dots + \ell_d x_d : 0 \leq \ell_i < L_i\}\] with $d \leq C_1(K)$ and $L_1 \cdots L_d \leq C_2(K) |A|$. The best known bounds for $C_1(K)$ and $C_2(K)$ are of shape $C_1(K) \sim K^{O(1)}$ and $C_2 \sim \exp(K^{O(1)})$, and furthermore simple examples show that these bounds cannot be improved other than by refining the $O(1)$ terms. See \cite{sanders-survey} for a comprehensive discussion. ``Typically'', a set $A$ of the form just described will have doubling constant something like $2^{C_1(K)} C_2(K)$, that is to say exponential in $K$. That is, a fair amount of information is lost in applying Freiman's theorem. 

To accurately count sets with doubling at most $K$, then, one must go beyond Freiman's theorem. In this paper, we succeed in doing this in certain ranges. Our main result is the following. 

\begin{thm}\label{thm:counting}
Fix $\delta > 0$ and $K > 0$. Then the following hold for all integers $k \geq k_0(\delta, K)$.
\begin{enumerate}
\item[$(i)$] For any $N \in \N$ there are at most
\[ 2^{\delta k} \binom{\frac{1}{2}Kk}{k} N^{\lfloor K + \delta \rfloor}\]
sets $A \subset [N]$ with $|A| = k$ and $|A + A| \leq K|A|$. 
\item[$(ii)$] If $N$ is prime there are at most 
\[ 2^{\delta k} \binom{\frac{1}{2}Kk}{k} N^{\lfloor K + \delta \rfloor}\]
sets $A \subset \Z/N\Z$ with $|A| = k$ and $|A + A| \leq K|A|$, provided that $Kk \leq (1 - \delta)N$.
\end{enumerate}
Corresponding bounds hold in either case if the sumset is replaced by the restricted sumset.\end{thm}

Two key examples show that both of the main terms $\binom{\frac{1}{2}Kk}{k}$ and $N^{K}$ in (i) are more-or-less necessary.

\emph{Example 1.} If $P$ is an arithmetic progression of length $Kk/2$ then any set $A \subset P$ of size $k$ will have $|A + A| \leq |P + P| < Kk$, and thus will have doubling less than $K$. For any fixed progression $P$ there are $\binom{\frac{1}{2}Kk}{k}$ such sets.

\emph{Example 2.} Let $A$ consist of an arithmetic progression $P$ of length $k - K + 2$ and $x_1,\dots, x_{K-2}$ are distinct elements, disjoint from $P$. There are almost $N^K$ such sets ($\sim N^2$ choices for $P$ and then $\approx N$ choices for each $x_i$). Furthermore 
\[ A +A = (P + P) \cup \bigcup_{i = 1}^{K-2} (x_i + P) \cup \bigcup_{i,j} \{x_i + x_j\},\] and so 
\[ |A + A| \leq 2|P| + (K-2)|P| + \textstyle\frac{1}{2}\displaystyle(K-2)(K-1) < K(|P| + K-2) = K|A|,\]
and hence $A$ has doubling less than $K$.

In the light of these examples, we see that Theorem~\ref{thm:counting} (i) is sharp when $k/\log N \rightarrow \infty$ (in which case the contribution from Example 1 dominates) and when $k/\log N \rightarrow 0$ (in which case Example 2 dominates). In the intermediate range $k \sim \log N$ our result is not sharp. If $k = \alpha \log N$ then presumably there are $N^{f(\alpha,K) + o(1)}$ sets $A \subset [N]$ with $|A| = k$ and $|A + A| \leq K|A|$ for some function $f(\alpha, K)$, but our results do not give this. We leave the exploration of this range as an open question. 

Note also that some condition such as $Kk \leq (1 - \delta)N$ is necessary in part (ii) of Theorem~\ref{thm:counting}. 
Indeed, if $k = \frac{N}{K}$ then, since we always have $A+A \subset \Z/N\Z$, the bound in (ii) would have to be modified to the much larger quantity $\binom{N}{k} = \binom{Kk}{k}$. It is also critical in this part that $N$ be prime. For example, if $N$ is even then $\Z/N\Z$ will have about $\binom{2k}{k}$ subsets $A$ with $|A| = k = \frac{1}{4}N$ and $|A+ A| \leq 2k$, namely all $k$-element subsets of the index 2 subgroup of $\Z/N\Z$.

We will apply Theorem~\ref{thm:counting} to the study of \emph{random Cayley graphs}. Let $\Gamma$ be a finite abelian group of size $N$, and select a set $A \subset \Gamma$ at random by choosing each $x \in \Gamma$ to lie in $A$ independently and at random with probability $\frac{1}{2}$. The \emph{Cayley sum graph} $G_A$ on vertex set $\Gamma$ is obtained by joining $x$ to $y$ if and only if $x + y \in A$. Such graphs are commonly considered as possible examples of highly-pseudorandom graphs. A frequently-considered example is the \emph{Paley sum graph} in which $N$ is prime and $A$ consists of the quadratic residues modulo $N$. To say that a graph is pseudorandom implies that it shares characteristics with the truly random Erd\H{o}s--R\'enyi graph $G(n,\frac{1}{2})$. One statistic one might look at is the \emph{clique number}, which is $(2 + o(1))\log_2 N$ for the Erd\H{o}s--R\'enyi graph. It is suspected that the Paley sum graph has clique number $O(\log^{1 + o(1)} N)$, but no bound better than $O(\sqrt{N})$ has been proven or seems likely to be at any time soon. (Note, however, that by a result of Graham and Ringrose \cite{graham-ringrose} the clique number of the Paley sum graph is not always $O(\log N)$; see \cite{G05} for further remarks on this point.)

If one is content with existence proofs, rather than explicit constructions such as the Paley sum graph, it is possible to do much better.  Our second result improves a theorem of the first author~\cite{G05}, and shows that there \emph{exist} Cayley sum graphs whose clique size essentially matches that of the Erd\H{o}s--R\'enyi graph.

\begin{thm}\label{maincliquetheorem}
For every $\eps > 0$, the following holds for all sufficiently large primes $N$. If $A \subset \Z/N\Z$ is chosen uniformly at random then, with probability $1 - o(1)$, the Cayley graph $G_A$ has no clique of size greater than $(2 + \eps) \log_2 N$.
\end{thm}

In \cite[Theorem 7]{G05} a similar result was obtained, but with $160$ replacing $2 + \eps$. Our argument builds on the argument there, but requires Theorem~\ref{thm:counting} as well as some other innovations of a rather different nature (see Section~\ref{sec5-dimensionbound}). Since Cayley graphs are regular, our argument shows that there are regular graphs attaining (essentially) the Erd\H{o}s--R\'enyi bound. Unsurprisingly this is a known result: see~\cite{KSVW}. 

A second application of Theorem~\ref{thm:counting} is the following result, which seems quite natural to us but does not appear to have been established before. Here $\N = \{1,2,3,\dots\}$. 

\begin{thm}\label{sumset-compliment}
Let $A \subset \N$ be chosen by selecting each positive integer to lie in $A$ with probability $\frac{1}{2}$, these choices being made independently. Then the probability that $A + A$ omits exactly $s$ elements of $\N$ is equal to $(2 + o(1))^{-s/2}$ as $s \rightarrow \infty$. 
\end{thm}

A similar result could be stated in which elements of $A$ are chosen with probability $p$, but we do not do so here. \vspace{11pt}

\textsc{Prior work.}  Let us conclude this introduction by stating more precisely the results of~\cite{ABMS,G05}, mentioned earlier, which motivated the research described in this paper. The following conjecture was made in~\cite{ABMS}.

\begin{conjecture}[Alon, Balogh, Morris and Samotij]\label{S+Sconj}
For every $\delta > 0$, there exists $C > 0$ such that the following holds. If $k \geq C \log N$ and if $K \leq k/C$, then there are at most
$$2^{\delta k} \binom{\frac{1}{2} Kk}{k}$$
sets $A \subset [N]$ with $|A| = k$ and $|A + A| \le K|A|$. 
\end{conjecture}

It is easy to see that Theorem~\ref{thm:counting} establishes this when $K = O(1)$. In~\cite[Theorem~1.3]{ABMS} the authors obtained bounds which are worse by a factor of roughly $(4/3)^k$, but which hold uniformly for every $K = o(k)$. The bound of~\cite[Proposition~23]{G05} (if adapted to subsets of $\{1,\dots,N\}$ instead of $\Z/N\Z$, which avoids the loss of a factor of $2$ from a certain lifting argument) is still weaker (by a factor of $(3/2)^k$); however, unlike either~\cite{ABMS} or the present paper, the arguments there give non-trivial bounds on the number of sets $A \subset [N]$ with $|A| = k$ and $|A + A| \le m$ for arbitrary $N$, $m$ and $k$.

The rest of the paper is structured as follows. In Section~\ref{sec2-regularity} we state our main technical theorem, a decomposition of $A$ into random-like intervals, and the regularity and counting lemmas which imply it. In Section~\ref{sec3-smalldoubling} we use this theorem, together with various basic facts about Freiman dimension and the group $\Z / N\Z$, to deduce Theorem~\ref{thm:counting}. In Section~\ref{sec4-furtherconsequences} we deduce Theorem~\ref{sumset-compliment} from Theorem~\ref{thm:counting}, and in Section~\ref{sec5-dimensionbound} we use the isoperimetric inequality on $\Z^d$ to bound the number of sets with sumset of size $c|A|^2$. Finally, in Section~\ref{sec6-cliquenumber}, we put the pieces together and deduce Theorem~\ref{maincliquetheorem}. The paper is concluded with an appendix in which we prove the regularity and counting lemmas.

\section{A structural decomposition of subsets of \texorpdfstring{$\Z / p\Z$}{Z/pZ}}\label{sec2-regularity}

In this section we introduce our main tool in the proof of Theorem~\ref{thm:counting}. It is based on a regularity lemma, essentially due to Ruzsa~\cite{GRprime} and the first author, which (roughly) states that every set $A \subset \Z/p\Z$ has a ``granular'' structure: after dilating $A$, we may partition $\Z/p\Z$ into intervals of length $L \rightarrow \infty$ such that, in a certain sense, $A$ behaves like a random set on each block. In particular, writing $A^*_i$ for the intersection of the dilated set $A^*$ with the $i^{th}$ interval of length $p/q$,
$$I_i(q) \, = \, \Big\{ x \in \Z/p\Z \,:\, x/p \in \big[i/q, \, (i+1)/q \big] \Big\},$$
we shall be able to show that, for almost all pairs $(i,j) \in [q]^2$, either one of $A^*_i$ and $A^*_j$ is very small, or $A_i^* + A_j^*$ is very large. 

\begin{thm}\label{reg:cor}
For every $\eps > 0$, there exists $\delta = \delta(\eps) > 0$ such that the following is true. Let $p > p_0(\eps)$ be a sufficiently large prime and let $A \subset \Z/p\Z$ be a set. There is a dilate $A^* = \lambda A$ of $A$ and a prime $q$, $\frac{1}{\eps^{10}} \leq q \le  p^{1 - \delta}$, such that the following holds. If 
\begin{equation}\label{def:Aistar}
A^*_i \, := \, A^* \cap I_i(q)
\end{equation}
for each $i \in [q]$ then, for at least $(1 - \eps) q^2$ of the pairs $(i,j) \in [q]^2$, 
$$\min\big( |A^*_i|, |A^*_j| \big) \le \eps p/q \qquad \textup{or} \qquad \big| A^*_i + A^*_j \big| \ge (2 - \eps)p/q.$$
\end{thm}

Theorem~\ref{reg:cor} follows easily from a certain \emph{regularity lemma} and associated \emph{counting lemma}. We state these more general results now. They are of similar strength to \cite[Proposition~3]{GRprime}, but are formulated and proved slightly differently. We give the proof in the appendix.

\textsc{A regularity lemma.} Our main technical tool is a variant of the celebrated regularity lemma of Szemer\'edi, which is a result about graphs,  for subsets of $\Z/p\Z$. It is different to (and in some ways weaker than) the usual \emph{arithmetic regularity lemma} \cite{green-arithreg, gt-arithreg}, but it comes with much better bounds.

In order to state the regularity lemma, we need a couple of slightly technical definitions. If $I \subset \Z/p\Z$ is an interval and if $A \subset I$ is a set, we shall consider the \emph{balanced Fourier transform} $\hat{f}_A \colon \R / \Z \to \C$, which is defined to be\footnote{We shall suppress the dependence of $\hat{f}_A$ on $I$, since the intervals we shall consider will be disjoint. It will therefore always be clear that the sum is over the interval which contains $A$.} 
\[ 
\hat{f}_A(\theta) \, := \, \sum_{x \in I} \big( 1_A(x) - \alpha \big) e(x \theta),
 \]
where $\alpha = |A| / |I|$. Note that $e(t) = e^{2\pi i t}$, a standard notation in analytic number theory. 

\begin{defn}[$\eps$-regularity]\label{reg-def}
We say that a pair $(A,A')$ of subsets of $\Z/p\Z$ is \emph{$\eps$-regular} if for every $\theta \in \R/\Z$ we have either
$$|\hat{f}_A(\theta)| \leq \eps |I| \quad \textup{or} \quad |\hat{f}_{A'}(\theta)| \leq \eps |I'|$$ 
and furthermore if $\Vert \theta \Vert \leq \min\big( \frac{1}{\eps |I|}, \frac{1}{\eps|I'|} \big)$ we have both
$$|\hat{f}_A(\theta)| \leq \eps |I| \quad \textup{and} \quad |\hat{f}_{A'}(\theta)| \leq \eps |I'|.$$ 
\end{defn}

We remark that this definition will invariably be applied when $|I| \approx |I'|$. Note that $\Vert x \Vert$ means the distance of $x$ from the nearest integer, this being a well-defined function on $\R/\Z$. Given $A^*$ and $q$, let the sets $A_i^*$ be defined as in~\eqref{def:Aistar}. 

\begin{GRRL}\label{regularity-lemma}
For every $\eps > 0$ there exists $\delta = \delta(\eps) > 0$ such that the following holds for every sufficiently large prime $p$, and every $A \subset \Z/p\Z$. There is a dilate $A^* = \lambda A$ of $A$ and a prime $q$, $\frac{1}{\eps^{10}} \leq q \le p^{1-\delta}$, such that at least $(1 - \eps) q^2$ of the pairs $(A^*_i, A^*_j)$ are $\eps$-regular.
\end{GRRL}

We emphasize that the upper bound on $q$ is important, since it ensures that the lengths of the discrete intervals $I_i(q)$ tend to infinity. We remark also that the bounds on $\delta$ and $p$ are fairly reasonable; in fact, it is sufficient to take $\delta = 2^{-(1/\eps)^{O(1)}}$ and $p = 2^{2^{O(1/\eps)}}$. The lower bound $q \geq \frac{1}{\eps^{10}}$ is convenient for our applications; the proof could easily be modified to make it larger if need be.

The point of making the definition of $\eps$-regularity is that we may use the regularity lemma in conjunction with the following counting lemma. It states that if the pair $(A^*_i, A^*_j)$ is $\eps$-regular, and neither set is too small, then $A^*_i + A^*_j$ is almost all of $I_i(q) + I_j(q)$. 

\begin{Count}\label{counting lemma}
Let $\eps, L$ be positive parameters with $L > 16/\eps$.  Suppose that $I, I' \subset \Z/p\Z$ are intervals with $|I|, |I'| = L+ O(1)$. Suppose also that the pair of sets $A \subset I$ and $A' \subset I'$ is $\eps^7$-regular and that $|A|, |A'|\geq \eps L$. Then $|A + A'| \geq (2 - 8\eps) L$.
\end{Count}

Theorem~\ref{reg:cor} is an almost immediate corollary of the regularity and counting lemmas, the former being applied with $\eps$ replaced by $c \eps^7$ for a suitably small absolute constant $c > 0$. 

Proofs of the regularity and counting lemmas are given in the appendix.

\section{Counting sets with small sumset in \texorpdfstring{$[N]$}{N} and \texorpdfstring{$\Z/N\Z$}{Z/NZ}}\label{sec3-smalldoubling}

In this section we prove our main result about counting sets with small sumset, Theorem~\ref{thm:counting}. The strategy in all cases is the same basic one used in~\cite{G05}: we count Freiman isomorphism classes of sets $A$ with $|A| = k$ and $|A + A| \leq Kk$, and then count the number of Freiman homomorphisms of $A$ into $[N]$ or $\Z/N\Z$. Recall that two sets $A$ and $B$ are said to be \emph{Freiman isomorphic} if there exists a bijection $f \colon A \to B$ such that $a + b = c + d$ if and only if $f(a) + f(b) = f(c) + f(d)$, for every $a,b,c,d \in A$.  Much more on Freiman homomorphisms may be found in~\cite{G05,TV} (for example). 

A key observation of Ruzsa, established below in a form suitable for our purposes, is that every set $A$ of integers with $|A + A| \leq K|A|$ has a ``dense model'', that is to say a Freiman isomorphic copy sitting inside some cyclic group $\Z/p\Z$ as a fairly dense subset. We begin by studying this ``dense'' situation (which is actually a special case of Theorem~\ref{thm:counting}).

\begin{prop}\label{mod:p:count}
If $\delta > 0$ and $p > p_0(\delta)$ is a sufficiently large prime, then the following holds for every $k,m \in \N$ with $\delta p \leq k < m \leq (1 - \delta) p$. There are at most $2^{\delta p} \binom{m/2}{k}$ sets $A \subset \Z / p\Z$ with $|A| = k$ and $|A + A| \leq m$.
\end{prop}

Proposition~\ref{mod:p:count} follows easily from Theorem~\ref{reg:cor}, together with the following lemma, an easy consequence of a result of Pollard~\cite{Pollard}. It is a kind of ``stability Cauchy-Davenport theorem''.

\begin{lemma}[Pollard]\label{pollard-lem}
Let $\beta > 0$, let $q > 16/\beta^2$ be a prime, and let $S$ be a subset of $\Z/q\Z$. Then there are at least $\min(2|S|,q) - \beta q$ elements of $\Z/q\Z$ having at least $\frac{1}{8}\beta^2 q$ representations as $i + j$ with $(i,j) \in S \times S$. 
\end{lemma}

\begin{proof}
Pollard in fact proved the following: if $N_t$ is the number of elements $x \in \Z/q\Z$ with at least $t$ representations as $i + j$ with $(i,j) \in S \times S$ then 
\[ 
\frac{N_1 + \dots + N_t}{t} \, \geq \, \min(2|S|,q) - t.
\] 
Now $N_i \leq q$ for all $i$, and obviously $N_1 \geq N_2 \geq \dots \geq N_t$. Therefore we have
\[  
\frac{\beta q}{2} + N_{\lfloor \beta t/2\rfloor} \, \geq \, \frac{\lfloor \beta t/2\rfloor q}{t} +  N_{\lfloor \beta t/2\rfloor} \, \geq \, \frac{N_1 + \dots + N_t}{t} \, \geq \, \min(2|S|,q) - t.
\]
Choosing $t = \lfloor \frac{1}{2}\beta q\rfloor$ we get
\[ N_{m} \geq \min(2|S|,q) - \beta q,\] 
where $m = \lfloor \frac{1}{2}\beta \lfloor \frac{1}{2}\beta q\rfloor \rfloor$. A short calculation confirms that $m \geq \frac{1}{8} \beta^2 q$ if $q > 16/\beta^2$.
\end{proof}

We can now deduce Proposition~\ref{mod:p:count} from Theorem~\ref{reg:cor} and Lemma~\ref{pollard-lem}.

\begin{proof}[Proof of Proposition~\ref{mod:p:count}]
Let $\delta > 0$, and assume without loss of generality that $\delta$ is sufficiently small. Let $p \ge p_0(\delta)$ be a sufficiently large prime, set $\beta = \delta^2/4$ and $\eps = \beta^3/8$, and let $k,m \in \N$ with $\delta p \le k < m \le (1 - \delta)p$. Choose an arbitrary set $A \subset \Z / p\Z$ with $|A| = k$ and $|A + A| \leq m$, and apply Theorem~\ref{reg:cor}. Since $p$ was chosen sufficiently large, we obtain a dilate $A^* = \lambda A$ of $A$ and a prime $q$, with $\frac{1}{\eps^{10}} \leq q \leq \beta p$, such that, for at least $(1 - \eps) q^2$ of the pairs $(i,j) \in [q]^2$, either
\begin{equation}\label{eq:regularity:AiAj}
\min\big( |A^*_i|, |A^*_j| \big) \le \eps L \qquad \textup{or} \qquad \big| A^*_i + A^*_j \big| \ge \big(2 - \eps \big) L,
\end{equation}
where $A^*_i$ is as defined in~\eqref{def:Aistar} and $L = p/q$. Let 
$$S \, := \, \Big\{ i \in \Z/q\Z \,:\, |A^*_i| > \eps L \Big\},$$
and note that, since $k \ge \delta p$, it follows that $|S| \ge \delta q / 2$. By Lemma~\ref{pollard-lem}, there is a set $T \subset \Z/q\Z$, containing at least $\min(2|S|,q) - \beta q$ elements of $\Z/q\Z$, such that every $t \in T$ has at least $\beta^2 q / 8$ representations as $i + j$ with $(i,j) \in S \times S$. 

Now, let us say that a pair $(i,j) \in S \times S$ is \emph{good} if~\eqref{eq:regularity:AiAj} holds, and recall that all but at most $\eps q^2$ pairs of $S \times S$ are good. Note that, as a consequence of \eqref{eq:regularity:AiAj}, if $(i,j) \in S \times S$ is good then
\begin{equation}\label{consequence} 
\big| ( A^*_i + A^*_j ) \cap I_{i+j}(q) \big| \, \geq \, ( 1 - \eps ) L.
\end{equation}
Let $T_* \subset T$ consist of those elements with at least one representation as $i + j$ with $(i,j)$ good. We claim that $|T_*| \geq \min(2|S|,q) - \big( \beta +  \frac{8\eps}{\beta^2} \big)q$. Indeed, this follows since $|T| \geq \min(2|S|, q) - \beta q$, all but at most $\eps q^2$ pairs of $S \times S$ are good, and every element $t \in T$ has at least $\beta^2 q / 8$ representations as $i + j$ with $(i,j) \in S \times S$. Writing $\tilde A := \bigcup_{i \in S} A^*_i$, it follows from this observation, together with~\eqref{consequence}, that
\[ m \, \ge \, |A + A| \, \geq \, |\tilde A + \tilde A|  \, \geq \, \big(\min(2|S|,q) - 2\beta q \big) ( 1 - \eps ) L \, > \, \big(1 - \delta \big)\min\big( 2|S|, q \big) L.
\]
since $\eps = \beta^3/8$ and $|S| \ge \delta q / 2 \ge 2 \beta q / \delta$. Now, if $2|S| \ge q$ then we obtain $m > \big( 1 - \delta \big) p$, which contradicts our assumption. It follows that $m \geq 2 \big(1 - \delta \big) |S| L$, and hence
\[ 
\bigg| \bigcup_{i \in S} I_i(q) \bigg| = |S|  L \, \le \, \bigg( \frac{1}{2} + \delta \bigg) m.
\]
Finally, observe (from the definition of $S$) that $A^*$ is composed of a subset of $\bigcup_{i \in S} I_i(q)$ together with at most $\eps p$ extra points. Moreover, the number of choices for the set $S \subset \Z/q\Z$ is at most $2^q$, and $A$ is simply a dilate of $A^*$. Hence, the number of sets $A \subset \Z / p\Z$ with $|A| = k$ and $|A + A| \leq m$ is at most
\[
p 2^q \sum_{\ell = 0}^{\eps p} \binom{p}{\ell} {(1 + 2\delta)m/2 \choose k - \ell} \, \le \, 2^{\delta p} {m/2 \choose k},
\]
since $p 2^q \sum_{\ell = 0}^{\eps p} \binom{p}{\ell} < 2^{\delta^2 p}$ and $\delta > 0$ is sufficiently small. The proposition follows.
\end{proof}

\vspace{11pt}

\textsc{Dense models.}  We turn now to a discussion of the ``dense models'' briefly alluded to in the introduction to this section. Ruzsa showed that if $A \subset \Z$ is a set with $|A + A| \leq K|A|$ and if $p$ is a prime with $p \geq K^C|A|$ then there is a set $A' \subset A$, $|A'| \geq |A|/2$, which is Freiman isomorphic to a subset of $\Z/p\Z$. A defect of this result is the fact that we must pass to a set of size $|A|/2$. By applying a result of the first author and Ruzsa~\cite{GRrect} (which gives a simpler proof, with somewhat better bounds, for a result of Bilu, Lev and Ruzsa~\cite{blr}) we can remove this blemish at the expense of increasing $p$.

\begin{prop}\label{weak-full-ruzsa-model}
Suppose that $A \subset \Z$ is a set with $|A + A| \leq K|A|$, and let $p$ be a prime with $p \geq 2(32 K)^{12K^2}|A|$. Then $A$ is Freiman isomorphic to a subset of $\Z/p\Z$.
\end{prop}

\begin{proof}
It is proven in~\cite[Theorem 1.4]{GRrect} that $A$ is Freiman isomorphic to a subset of $[m] \subset \Z$, where $m \leq (32 K)^{12 K^2} |A|$. Now simply compose this Freiman isomorphism with the projection $\Z \rightarrow \Z/p\Z$, where $p$ is any prime greater than $2m$. \end{proof}

\vspace{11pt}

\textsc{Freiman dimension.} To understand the number of Freiman embeddings of a set $A$ into $[N]$ we use the concept of \emph{Freiman dimension}. Given a set $A$ in some abelian group $\Gamma$, we define the Freiman dimension $r_{\Q}(A)$  of $A$ to be $\dim(\Hom(A, \Q)) - 1$, where $\Hom(A, \Q)$ is the vector space over $\Q$ of Freiman homomorphisms $\phi \colon A \rightarrow \Q$.  The $-1$ is included so that an arithmetic progression such as $[n]$ has Freiman dimension 1 and not 2, and a single point has Freiman dimension $0$. The following proposition is essentially due to Freiman; for completeness we shall sketch the proof.

\begin{prop}\label{prop:freiman} 
Let $\Gamma$ be an abelian group, and let $A \subset \Gamma$. If $|A + A| \leq K|A|$, then the Freiman dimension $r = r_{\Q}(A)$ satisfies
\[  r \leq K - 1 + \frac{1}{|A|}\binom{r+1}{2}. \] 
\end{prop}

The proposition is a straightforward consequence of Freiman's Lemma, a proof of which may be found in, for example,~\cite[Proposition 3.3]{GEdin}.

\begin{FL}
Suppose that $A \subset \R^r$ is not contained in an affine subspace. Then
$$|A+A| \, \ge \, (r+1)|A| - {r+1 \choose 2}.$$ 
\end{FL}

We shall also use the following lemma, which follows immediately from~\cite[Lemma~13]{G05}.

\begin{lemma}\label{lem:basicG05}
Let $\Gamma$ be an abelian group, and let $A \subset \Gamma$ be a set with Freiman dimension $r = r_{\Q}(A)$. Then
\begin{itemize}
\item[$(a)$] There is a Freiman isomorphic image of $A$ in $\R^{r}$ which is not contained in any proper affine subspace.
\item[$(b)$] There are at most $N^{r+1}$ Freiman homomorphisms from $A$ into $[N]$. 
\end{itemize}
\end{lemma}


\begin{proof}
By~\cite[Lemma~13]{G05}, there exist elements $a_1,\dots, a_{r+1} \in A$ such that there is a unique Freiman isomorphism $\phi \colon A \rightarrow \R^{r+1}$ with $\phi(a_i) = e_i$ for each $i \in \{1,\ldots,r+1\}$. Moreover, for each $a \in A$ we have $\phi(a) = \lambda_1 \phi(a_1) + \dots + \lambda_{r+1} \phi(a_{r+1})$ for some $\lambda_i \in \R$ with $\sum_i \lambda_i = 1$. Hence $\phi(A)$ lies in the affine subspace of $\R^{r+1}$ given by $\lambda_1 + \dots + \lambda_{r+1} = 1$, which is isomorphic to $\R^r$, as required.

For part~$(b)$ note that, again by~\cite[Lemma~13]{G05}, once we have chosen where to map $a_1,\ldots,a_{r+1}$, the remaining elements of $\phi(A)$ are determined uniquely. There are thus at most $N^{r+1}$ chocies, as claimed.
\end{proof}

We can now easily deduce Proposition~\ref{prop:freiman}.

\begin{proof}[Proof of Proposition~\ref{prop:freiman}]
By Lemma~\ref{lem:basicG05}$(a)$ there is a Freiman isomorphic image of $A$ in $\R^{r}$ which is not contained in any proper affine subspace. By Freiman's Lemma, we have
\[ K|A| \geq |A + A| \geq (r+1)|A| - \binom{r+1}{2}.\] 
The result follows immediately. 
\end{proof}


\textsc{Two lemmas relevant to $\Z/N\Z$.} The proof of part (ii) of Theorem~\ref{thm:counting}, which concerns subsets of $\Z/N\Z$, is a little trickier than the proof of part (i), which is concerned with subsets of $[N]$. To handle it we need two special lemmas, proven in this subsection. 

Our strategy will, roughly speaking, be to map each $A \subset \Z/N\Z$ to many sets of the form $\lambda A + \mu$ which are contained in the interval $\{1,\ldots,N/2\}$. It will then be straightforward to deduce the result for $\Z/N\Z$ from the corresponding result for $[N]$. The first lemma shows that there are many such maps. In the following $C$ denotes an absolute constant which could be specified if desired, though on some occasions this might be hard work. Different instances of the letter $C$ may denote different constants.

\begin{lemma}\label{many-rect} 
Suppose that $A \subset \Z/N\Z$ is a set of cardinality $k$ and that $|A + A| \leq Kk$. Then there are at least $\exp(-K^C) N / k$ dilates $\lambda A$ of $A$ that are contained in a subinterval of $\Z/N\Z$ of length less than $N/4$.
\end{lemma}

Note that in the lemma above we do not preclude $\lambda = 0$; however if $k > \exp(K^{-C})N$ this might be the only dilate with the required property, in which case the result has no content. Without the requirement that there be \emph{many} dilates $\lambda A$, this result follows from~\cite[Theorem 2.1]{blr}. The paper~\cite{blr} uses Freiman's theorem (stated in the introduction), but a more self-contained proof of~\cite[Theorem 2.1]{blr} was obtained by the first author and Ruzsa~\cite{GRrect}. Either of these arguments may be adapted to give the stronger statement we require here, but (given Freiman's theorem) a somewhat shorter account can be given by following the original approach of~\cite{blr}. We leave it as an exercise to the reader to adapt the proof in~\cite{GRrect}; see also~\cite[Lemma 5.45]{TV}.

\begin{proof}[Proof of Lemma~\ref{many-rect}]
We first apply Freiman's theorem with Chang's bounds~\cite{Chang} to conclude that $A$ is contained in a generalised progression
\[ 
P = \big\{ \ell_1 x_1 + \dots + \ell_d  x_d : 0 \leq \ell_i < L_i \big\}
\] 
where $x_1,\dots,x_d \in \Z/N\Z$, $d \leq K^C$ and $\prod_{j = 1}^d L_j \leq \exp(K^C) k$. Strictly speaking, Freiman's theorem concerns subsets of $\Z$, and not of $\Z/N\Z$. However (see~\cite[proof of Theorem 2.1]{blr}), we may lift $A$ to a subset $\overline{A} \subset \Z$ of cardinality $k$ and with doubling constant at most $2K$, apply Freiman's theorem to this, and then push the resulting progression back down to $\Z/N\Z$. 

Let $\Vert \cdot \Vert$ denote distance to the nearest integer. We claim that if $\lambda \in \Z/N\Z$ satisfies 
\begin{equation}\label{eq:lambdanearintegers}
\left\Vert \frac{\lambda x_i}{N} \right\Vert \, \leq \, \frac{1}{8d L_i} \qquad \mbox{ for each } i \in \{1,\dots, d\}, 
\end{equation}
then $\lambda A$ is contained in a subinterval of $\Z/N\Z$ of length less than $N/4$. Indeed, this follows easily from the triangle inequality and the fact that $A \subset P$, since for any such $\lambda$ and for any $a \in A$ we have $\Vert \lambda a/N \Vert \leq \frac{1}{8}$.  

It remains to show that many $\lambda \in \Z/N\Z$ satisfy~\eqref{eq:lambdanearintegers}. To see this, observe that, by an averaging argument (the pigeonhole principle), there exists a translate of the box $B = \prod_{i = 1}^d [0, 1/8dL_i] \subset (\R/\Z)^d$ containing at least $\vol(B) N$ points $(tx_1/N,\dots, t x_d/N) \in (\R/\Z)^d$ with $t \in \Z/N\Z$. Writing $T \subset  \Z/N\Z$ for the set of such $t$, it follows that every $\lambda \in T - T$ satisfies~\eqref{eq:lambdanearintegers}, and hence the number of such $\lambda$ is at least $\vol(B) N$. Since
\[ \vol(B) \, = \, (8d)^{-d} \frac{1}{L_1 \cdots L_d} \, \geq \, \frac{\exp\big( K^{-C'} \big)}{k},\] 
the claimed bound follows.
\end{proof}

The next lemma shows that only a bounded\footnote{Note that here, as throughout, the notation $f(x) \ll g(x)$ denotes the existence of an absolute constant $C$ such that $f(x) \le C g(x)$ for every $x$.} number of maps $\lambda A + \mu$ give the same set. 

\begin{lemma}\label{invariance-prop}
Suppose that $A \subset \Z/N\Z$ is a set of cardinality $k$ and that $|A + A| \leq Kk$. Then either $|A| \gg K^{-50}N$, or else there are $\ll K^{100}$ values of $\lambda \in (\Z/N\Z)^{\times}$ and $\mu \in \Z/N\Z$ such that $A = \lambda A + \mu$.
\end{lemma}

\begin{proof}
First note that for fixed $\lambda$ at most one value of $\mu$ is permissible. Indeed if $\lambda A + \mu = \lambda A + \mu'$ then $\lambda A$ is invariant under the additive group $\langle \mu - \mu' \rangle$, which is all of $\Z/N\Z$ since $N$ is prime. Thus the task reduces to counting the number of possible $\lambda$. 

Suppose then that $\lambda \neq 1$, and that $A = \lambda A + \mu$ for some $\mu$. Then writing $A' = A + \frac{\mu}{\lambda - 1}$ we have $A' = \lambda A'$, and hence $A' = HA'$ where $H = \langle \lambda \rangle \leq (\Z/N\Z)^{\times}$ is the subgroup generated by $\lambda$. Applying the asymmetric sum-product estimate of Bourgain~\cite[Theorem 6]{bourgain-multilinear} (see also the short and quantitatively crisp article of Garaev~\cite{garaev-quant}) we have either $|H| \ll K^{50}$ or $|A| \gg K^{-50}N$. If the second inequality holds then we are done, so let us assume that $|H| \ll K^{50}$. Then $\lambda^m = 1$ for some $m \ll K^{50}$. For each $m$, the number of such $\lambda$ is no more than $m$, and hence the total possible number of $\lambda$ is, crudely, bounded by $\sum_{m \ll K^{50}} m  \ll K^{100}$.  This concludes the proof.
\end{proof}

In the proof of Theorem~\ref{thm:counting} we shall need one more fact, proved by Schoen~\cite{Schoen}, which we shall use to deduce the bounds for sets with small restricted sumset. We remark that the proof uses Roth's theorem on 3-term progressions.

\begin{lemma}\label{lem:restricted}
Let $A$ be a subset of $\Z$ or $\Z/N\Z$ with $|A| = k$. Then
$$|A \plus A| \, = \, \big( 1 + o(1) \big) |A + A|$$
as $k \to \infty$. 
\end{lemma}

\begin{proof}[Proof of Theorem~\ref{thm:counting}] 
We turn now to the proof of our main result concerning counting sets with small sumset. We begin by looking at the case of sets $A \subset [N]$, that is to say part (i) of the theorem. By Proposition~\ref{weak-full-ruzsa-model}, it follows that any $A \subset [N]$ with $|A| = k$ and $|A + A| \leq Kk$ is Freiman-isomorphic to a subset of $\Z/p\Z$ with $p = O_K(k)$. Hence, by Proposition~\ref{mod:p:count}, the number of isomorphism classes of such $A$ is $2^{o_K(k)} \binom{Kk/2}{k}$ as $k \to \infty$.

Let $\overline{A}$ be a representative of such an isomorphism class, and suppose $r_{\Q}(\overline{A}) = r$. By Lemma~\ref{lem:basicG05}$(b)$, the number of Freiman homomorphisms $\pi : \overline{A} \rightarrow [N]$ is at most $N^{r+1}$. Since each set $A$ of the type we are counting must be the image of such a homomorphism, we get
\[ 2^{o_K(k)} \binom{Kk/2}{k} N^{r+1}\] such sets in total.

Now, we showed, in Proposition~\ref{prop:freiman} above, that $r \leq K  - 1 + \frac{1}{|A|} \binom{r+1}{2}$. Applying this with the trivial bound $r \leq |A|$, we deduce that $r \leq 2K - 1$. Feeding this back in, we infer that $r \leq K - 1 + \frac{2K^2}{|A|}$. Hence, if $|A| = k \ge 2K^2 / \delta$, then $r \leq K - 1 + \delta$. Since $r$ is an integer, we in fact have $r +1 \leq \lfloor K + \delta\rfloor$. The claimed bound now follows.

We next turn to sets $A \subset \Z/N\Z$, i.e. part (ii) of the theorem. Suppose that $|A| = k$ and that $|A + A| \leq Kk$. 
By Lemma~\ref{many-rect} there are at least $c_K N/k$ dilates $\lambda A$ that are contained in a subinterval of $\Z/N\Z$ of length $N/4$, where $c_K$ depends only on $K$ (and in fact can be taken to be $\exp(K^{-O(1)})$.

If $k \geq \frac{1}{2} c_K N$ then Theorem~\ref{thm:counting} follows directly from Proposition~\ref{mod:p:count}. Suppose, then, that $k < \frac{1}{2} c_K N$. Then at least $c_K N/2k$ of the dilates $\lambda A$ supplied by the application of Lemma~\ref{many-rect} have $\lambda \neq 0$. For each of these dilates there are at least $N/4$ translates $\lambda A + \mu$ that are contained in the interval $\{1,\dots, N/2\}$. Assuming (as we may) that $c_K < K^{-50}$, Lemma~\ref{invariance-prop} tells us that there are $\gg c'_K N^2/k$ distinct such sets. Each of these may be lifted to a subset of $[N]$ that is Freiman isomorphic to $A$, and hence have the same cardinality and doubling as $A$. 

Thus every set $A \subset \Z/N\Z$ with $|A| = k$ and $|A+A| \le K|A|$ arises in at least $c'_K \frac{N^2}{k}$ ways by taking a set $\overline{A} \subset [N]$ with $|\overline{A}| = k$ and $|\overline{A} + \overline{A}|  \leq Kk$, projecting it modulo $N$ and then applying the inverse of an affine map $x \mapsto \lambda x + \mu$ with $\lambda \neq 0$.  There being fewer than $N^2$ such affine maps, it follows that the number of $A \subset \Z/N\Z$ with $|A| = k$ and $|A+A| \leq K|A|$ is at most $k/c'_K$ times the number of subsets of $[N]$ with the same properties, and so the bound we seek follows from part~(i).

Finally, we address the claim that all of our bounds hold equally well when we are counting sets $A \subset [N]$ (or $A \subset \Z/N\Z$) with $|A| = k$ whose \emph{restricted} sumset $A \hat{+} A$ has size at most $K|A|$. Indeed, by Lemma~\ref{lem:restricted}, it follows that in fact $|A + A| \leq  K'|A|$ with $K' = K + o(1)$. This error of $o(1)$ may be absorbed into the error term in our main theorem, and so the same bounds hold for sets with small restricted sumset, as claimed. 
\end{proof}

\section{Further consequences}\label{sec4-furtherconsequences}

In this section we begin by proving Theorem~\ref{sumset-compliment}. Let us begin by recalling the statement.

\begin{sumset-compliment-repeat}
Let $A \subset \N$ be chosen by selecting each positive integer with probability $1/2$, all independently. Then $\Pr\big( | \N \setminus (A + A) | = s \big) = \big( 2 + o(1) \big)^{-s/2}$ as $s \rightarrow \infty$. 
\end{sumset-compliment-repeat}

This is, in actual fact, a relatively easy deduction from Theorem~\ref{thm:counting}. Suppose that $A \subset \N$ is selected at random. Write $X := A \cap [10s]$, and note that, with very high probability,
\[ 
\N \setminus (A + A) = \{1,\dots, 10s\} \setminus (X + X).
\]
Indeed the only way this could fail to be the case is if $A+A$ omits some integer $n > 10s$. If $n \notin A + A$ then $A$ contains at most one element from each of the pairs $\{1,n-1\}, \{2, n-2\}, \dots, \{\lfloor n/2\rfloor, n - \lfloor n/2 \rfloor\}$. These pairs being disjoint, the chance of this happening is bounded by $(3/4)^{\lfloor n/2\rfloor}$. Hence
\[ \P\big(\mbox{$A + A$ misses some $n > 10s$} \big) \leq \sum_{n > 10s} \left( \frac{3}{4} \right)^{\lfloor n/2\rfloor} < \, 2^{-s},\] a quantity we can ignore for the rest of the argument. 

The set $X$ is, of course, a (uniformly chosen) random subset of $[10s]$. We are interested, then, in counting the number of such subsets for which 
\[ 
\big| \big\{ 1,\dots, 10s \big\} \setminus \big( X + X \big) \big| = s.
\] 
We shall in fact count the slightly larger family of subsets for which $|X + X| \le 19s$. The aim is to show that this number is at most $2^{10s - s/2 + o(s)}$. 

\begin{proof}[Proof of Theorem~\ref{sumset-compliment}]
By Theorem~\ref{thm:counting}, the number of sets $X \subset \{1,\ldots,10s\}$ for which $|X + X| \le 19s$ is bounded above by
$$s {10s \choose s/10} + 2^{o(s)} \sum_{k = s/10}^{10s} \binom{19s/2}{k} \, \le \, 2^{19s/2 + o(s)},$$
as required. By the observations above, it follows that $a(s) := \Pr\big( | \N \setminus (A + A) | = s \big) \le \big( 2 + o(1) \big)^{-s/2}$ as $s \rightarrow \infty$. The corresponding lower bound is trivial since, considering only those sets with $1 \not\in A$, we have $a(s) \ge \frac{1}{2}a(s-2)$ for every integer $s \geq 3$. 
\end{proof}

\vspace{11pt}

We believe that the bound in Theorem~\ref{sumset-compliment} can be improved, potentially to $O(2^{-s/2})$. We plan to return to this issue in a future paper.

\section{The dimension of a subset of \texorpdfstring{$\Z/N\Z$}{Z/NZ}}\label{sec5-dimensionbound}

The aim of this section is to give an upper bound for the number of $k$-subsets $A \subset \Z/N\Z$ with $|A+A| \le m$ when $N$ is prime and $m = \Omega(k^2)$. This is needed in the proof of Theorem~\ref{maincliquetheorem}. The technique used in~\cite{G05} proceeded by unwrapping $\Z/N\Z$ to lift to $\Z$, which may cause the doubling constant to double. It is somewhat ineffectual in most of the range $m = \Omega(k^2)$ and completely ineffectual when $m > k^2/4$. We shall use a different approach, which yields the following improved bound.

\begin{prop}\label{prop:count:toprange}
Let $\eps > 0$, let $N$ be a large prime, and suppose that $k \leq 100 \log N$. If $m \geq \eps k^2$, then there are at most $N^{2m/k + o(k)}$ sets $A \subset \Z/N\Z$ with $|A| = k$ and $|A + A| \leq m$. 
\end{prop}

Note that if $A$ consists of an arbitrary set of $t = \lfloor cm/k \rfloor$ elements, together with an arithmetic progression of length $k-t$, then $|A+A| \le {t \choose 2} + t(k-t) + 2(k - t) \le m$ if $c \approx \big( 1 - \sqrt{1 - 2\alpha} \big) / \alpha \to 2$ as $\alpha = m / k^2 \to 1/2$, so the bound in the proposition is not far from optimal when $m$ is close to ${k \choose 2}$. More importantly for our purposes, it is sufficient to deduce Theorem~\ref{maincliquetheorem}. 

The idea behind the proof of Proposition~\ref{prop:count:toprange} is that our set $A$ cannot contain more than about $2m/k$ ``quasi-random" elements. The crucial definition, which quantifies the term quasi-random, is as follows.

\begin{defn}
A set $\{x_1,\dots,x_d\} \subset \Z/N\Z$ is \emph{$M$-dissociated} if $\lambda_1 x_1 + \dots + \lambda_d x_d \neq 0$
for every collection $\lambda_1,\ldots,\lambda_d \in \Z$ of integers, not all zero, with $\sum_i |\lambda_i| \leq M$. 
\end{defn}

The key step in our argument is the following statement.

\begin{lemma}\label{FDLmodN}
Let $\eps > 0$, let $N$ be a large prime, and suppose that $k$ and $m$ are integers with $k \leq 100 \log N$ and $m \ge \eps k^2$. Suppose that $A \subset \Z/N\Z$ is a set of size $k$ with $|A + A| \le m$, and that $\{x_1,\dots,x_d\} \subset A$ is $(\log N)^{7/8}$-dissociated. Then $d \, \leq \, \big( 1 + o_\eps(1) \big) \frac{2m}{k}$
as $k \to \infty$.
\end{lemma}

The proof of Lemma~\ref{FDLmodN} uses the isoperimetric inequality on $\Z^d$ (see below) and the following simple graph-theoretic lemma. If $G$ is a graph, then we write $V(G)$ for its set of vertices, and if $x, x' \in V(G)$ then we write $\dist(x,x')$ for the length of the shortest path from $x$ to $x'$, defining this to be $\infty$ if there is no such path (though in our examples there always will be). Write $B_r(x)$ for the ``ball'' consisting of all $x' \in V(G)$ with $\dist(x, x') \leq r$, and $\partial B_r(x)$ for the ``sphere'' containing all those $x'$ with $\dist(x, x') = r$ exactly. 

\begin{lemma}\label{graph-lemma}
Suppose that $G$ is a graph and that $A \subset V(G)$. Let $D > 1$ be a parameter. Then we may split $A$ as a disjoint union $A_1 \cup \dots \cup A_\ell \cup A_*$, where: 
\begin{enumerate}
\item[(a)] $|A_*| \leq 32(|A|/D)^2$.
\item[(b)] If $i \neq j$, no vertex of $A_i$ is joined in $G$ to any vertex of $A_j$.
\item[(c)] The diameter of each $A_i$ is at most $D$. 
\end{enumerate}
\end{lemma}

\begin{proof}
Note first that it suffices to prove the lemma for each connected component of $G[A]$, the subgraph of $G$ induced by $A$, since the function $x \mapsto x^2$ is convex. Hence, let us assume that $G[A]$ is connected, and define the sets $A_1,\ldots,A_\ell$ and $A_*$ as follows. First, let $X  = \{x_1,\dots, x_\ell\}$, be a maximal $D/4$-separated subset of $A$: that is to say, if $i \neq j$ then $\dist(x_i,x_j) \geq D/4$, and $X$ is maximal with respect to this property. Next, choose, for each $j \in [\ell]$, a radius $r(j) \in (D/4,D/2]$. Finally, set
$$A_j \, = \, B_{r(j)-1}(x_j) \setminus \bigcup_{i=1}^{j-1} B_{r(i)}(x_i)$$
for each $j \in [\ell]$, and set $A_* = V(G) \setminus \bigcup_{j=1}^\ell A_j$. 

We claim that, for some choice of radii, these sets satisfy properties (a), (b) and (c). Indeed, property (b) follows from the fact that $A_i \subset B_{r(i)-1}(x_i)$ and $A_j \cap B_{r(i)}(x_i) = \emptyset$ for every $i < j$, and property (c) holds since $r(j) \le D/2$ for each $j \in [\ell]$. To prove that $A_*$ satisfies property (c), note first that the balls $B_{D/4}(x_j)$ cover $V(G)$, by the maximality of $X$, and so
$$A_* \subset \bigcup_{j=1}^\ell \partial B_{r(j)}(x_j).$$
Moreover, the balls $B_{D/8}(x_i)$ are disjoint, and each contains at least $D/8$ vertices (since $G[A]$ is connected), so $\ell \le 8|A| / D$. Thus it only remains to give an upper bound on $|\partial B_{r(j)}(x_j)|$. To do so, we simply choose $r(j) \in (D/4,D/2]$ to minimize the size of $\partial B_{r(j)}(x_j)$; since the spheres are disjoint, it follows from the pigeonhole principle that $| \partial B_{r(j)}(x_j) | \le 4 |A| / D$. Hence
$$|A_*| \, \le \, \frac{4\ell |A|}{D} \, \le \, \frac{32 |A|^2}{D^2},$$
as claimed.
\end{proof}
We continue now with the proof of Lemma \ref{FDLmodN}. Let $A \subset \Z/N\Z$ be as in the statement of that lemma, and suppose that $\{x_1,\dots, x_d\} \subset A$ is $(\log N)^{7/8}$-dissociated. We will apply Lemma~\ref{graph-lemma} to a certain graph $G$ on vertex set $\Z/N\Z$, with two vertices joined by an edge if they differ by $\pm x_i \pm x_j$ for some $i,j$. By a (very minor) abuse of notation, regard $A$ as a subset of the set $V(G)$ of vertices of this graph.
Note that we have the following useful property: if $a \in A$, then every element of $A$ has the form $a + \sum_{j = 1}^d \lambda_j x_j$ with $\sum_j |\lambda_j|$ at most twice the diameter of $A$ in the graph.

By Lemma~\ref{graph-lemma} applied with $D = k^{1/2}$, it follows that there exists a partition $A = A_1 \cup \dots \cup A_\ell \cup A_*$ such that $|A_*| \ll k^{1/2} = o(|A|)$ as $k \to \infty$, each $A_i$ has diameter at most $k^{3/4} \ll (\log N)^{3/4}$, and no vertex of $A_i$ is adjacent (in $G$) to any vertex of $A_j$ when $i \neq j$. This last condition implies the following important property:  if $i \neq j$ then $A_i + \{x_1,\dots, x_d\}$ and $A_j + \{x_1,\dots, x_d\}$ are disjoint. In particular,
\begin{equation}\label{sumset-lower} |A + A| \geq \sum_{i = 1}^\ell |A_i + \{x_1,\dots, x_d\}|.\end{equation}
In order to bound the right-hand side of this equation, we shall use the following well-known vertex isoperimetric inequality on the grid, first proved by Wang and Wang~\cite{WW}, see also~\cite[Theorem~4]{BL}.

\begin{VII}\label{IsopIneq}
For every $\eps > 0$ and $C > 0$ there exists $d_0 = d_0(C,\eps) \in \N$ such that the following holds for every $d \ge d_0$. If $S \subset \Z^d$ is a set of size at most $C d$, then $\big| S + \{e_1,\dots, e_d\} \big| \, \geq \, ( \textstyle\frac{1}{2} - \eps ) d |S|$.
\end{VII}

\begin{proof} 
By translating we may replace $\Z^d$ with $\Z_{\geq 0}^d$.  The theorem of Wang and Wang states that the vertex boundary of a set $A \subset \Z_{\geq 0}^d$ of size $k$ is minimized by taking the first $k$ elements of the \emph{simplicial order}, that is, the order such that $\x < \y$ if either
$\sum_j x_j < \sum_j y_j$ or if $\sum_j x_j = \sum_j y_j$ and there is some $j$ such that $x_1 = y_1, \dots, x_{j-1} = y_{j-1}$ and $x_j > y_j$. 
Thus the simplicial order on $\Z_{\geq 0}^d$ begins $0, e_1, \dots, e_d, 2e_1, e_1 + e_2,\dots$, and so on. It is straightforward (e.g., via compression) to see that the extremal sets in the `oriented' case (i.e., when we only allow edges `to the right') are the same. 

Now, an initial segment of the simplicial order of size $C d$ is precisely
$$\{0\} \cup \big\{ e_1,\ldots,e_d \big\} \cup \big\{ e_i + e_j \,:\, i \leq j \textup{ and } i \leq \lfloor C\rfloor - 1 \big\} \cup \big\{ e_{\lfloor C \rfloor} + e_j \,:\, j < c' d \big\},$$ 
where $c' \ge \{C\}$, the fractional part of $C$. The vertex boundary of this set contains
$$\big\{ e_i + e_j \big\} \cup \big\{ e_i + e_j + e_k \,:\, i \leq j \leq k, i \leq \lfloor C\rfloor -1 \big\} \cup \big\{ e_{\lfloor C \rfloor} + e_j + e_k \,:\, j < \{C\} d \big\}.$$ A short computation shows that this set has size $\big( \frac{1}{2} \lfloor C \rfloor + \{C \} - \frac{1}{2} \{C\}^2 \big) d^2  - O_C(d)$. Since $0 \leq \{C\} \leq 1$, this is at least $\frac{1}{2} C d^2 - O_C(d)$ and so the claimed bound follows. 
\end{proof}

We remark that an upper bound on $|S|$ is certainly necessary, since if $S$ is the Hamming ball of radius $2$ then $|S| \sim d^2/2$ whilst $|S + \{e_1 , \dots , e_d \}| \sim d^3/6$. 

\begin{proof}[Proof of Lemma~\ref{FDLmodN}] 
Let $A = A_1 \cup \dots \cup A_\ell \cup A_*$ be the partition given by Lemma~\ref{graph-lemma}, and recall that, choosing $a_i \in A_i$ arbitrarily, every element $a \in A_i$ has the form $a_i + \sum_{j = 1}^d \lambda_j x_j$ with $\sum_j |\lambda_j| \ll (\log N)^{3/4}$. Define a map $\pi : A_i \rightarrow \Z^d$ by $\pi(a) = (\lambda_1,\dots, \lambda_d)$, and observe that $\pi$ is well-defined and furthermore is a Freiman isomorphism. To see this, recall that $\{x_1,\dots, x_d\}$ is $(\log N)^{7/8}$-dissociated, and thus if $\pi(a) + \pi(b) \ne \pi(c) + \pi(d)$ then $a + b \ne c + d$ if $N$ is sufficiently large. (The reverse implication is trivial.)

Applying the isoperimetric inequality to the set $\pi(A_i)$, which has size at most $k \le d / \eps$ (otherwise $d \le \eps k \le m / k$, in which case we are done), it follows that 
$$\big| A_i + \big\{ x_1,\dots, x_d \big\} \big| \, \geq \, \left( \textstyle\frac{1}{2}\displaystyle - o_\eps(1) \right) d |A_i|$$
as $k \to \infty$. Since $\sum_{i = 1}^k |A_i| = |A| - |A_*| = \big( 1 - o(1) \big) |A|$ as $k \to \infty$, it follows that
\[ 
|A + A| \, = \, \sum_{i = 1}^\ell \big| A_i + \big\{ x_1,\dots , x_d \big\} \big|\, \geq \, \left(\textstyle \frac{1}{2} \displaystyle- o_\eps(1) \right) d \sum_{i = 1}^\ell |A_i| \, \geq \, \left( \textstyle\frac{1}{2}\displaystyle - o_\eps(1) \right) d |A|,
\]
as required.
\end{proof}

Before proving Proposition~\ref{prop:count:toprange} we note a simple bound for the number of lattice points in an octahedron.

\begin{lemma}\label{ell-1-bound}
For every $M, d \in \N$, 
the number of $d$-tuples of integers $(\lambda_1,\dots, \lambda_d)$ with $\sum_{i = 1}^d |\lambda_i| \leq M$ is bounded from above by $(4d)^M$.
\end{lemma}
\begin{proof}
It is well-known and easy to see that the number of \emph{positive} such tuples, in which $\lambda_i \geq 0$ for all $i$, is precisely $\binom{M+d}{M} \leq (2d)^M$. Every $d$-tuple with $\sum |\lambda_i| \leq M$ can be turned into a positive one by flipping at most $M \le 2^M$ signs. \end{proof}

\begin{proof}[Proof of Proposition~\ref{prop:count:toprange}]
Given a set $A \subset \Z/N\Z$ with $|A| = k$ and $|A + A| = m$, where $k \leq 100 \log N$ and $m \geq \eps k^2$, let $X = \{x_1,\dots, x_d\} \subset A$ be a maximal $(\log N)^{7/8}$-dissociated set. By Lemma~\ref{FDLmodN}, we have $d \, \leq \, \big( 1 + o_\eps(1) \big) \frac{2m}{k}$
as $k \to \infty$. 

Now, the number of choices for $x_1,\dots, x_d$ is clearly at most $N^d$. Moreover, for each such choice, it follows from the maximality of $X$ that for every element $a \in A$ there is a relation with integer coefficients $\lambda a + \sum_{i = 1}^d \lambda_i x_i = 0$ with $|\lambda| + \sum_i |\lambda_i| \leq (\log N)^{7/8}$ and $\lambda \neq 0$. Thus, once $x_1,\dots, x_d$ have been selected, it follows from Lemma~\ref{ell-1-bound} that $a$ is chosen from a set of size at most 
$$(\log N)^{7/8} \big( 4d \big)^{(\log N)^{7/8}} \, = \, N^{o(1)},$$
since trivially $d \leq k \leq 100 \log N$. Thus, for a fixed choice of $x_1,\dots, x_d$, the number of possibilities for $A$ is at most $N^{o(k)}$. The claimed bound now follows immediately.
\end{proof}

\section{The clique number of random Cayley graphs on \texorpdfstring{$\Z/N\Z$}{Z/NZ}}\label{sec6-cliquenumber}

In this section we use the method  of~\cite{G05} together with Theorem~\ref{thm:counting} and Proposition~\ref{prop:count:toprange} in order to deduce Theorem~\ref{maincliquetheorem}, the statement that the clique number of a random Cayley graph on $\Z/N\Z$ is $(2 + o(1))\log_2 N$.

We begin by recalling some simple remarks from~\cite{G05}. Given an abelian group $\Gamma$, recall that $A \plus A = \big\{ a + b : a,b \in A, \, a \neq b \big\}$ denotes the the \emph{restricted} sumset of $A \subset \Gamma$, and write 
$$S_k^m(\Gamma) \, = \, \big\{ A \subset \Gamma \,:\, |A| = k \textup{ and } |A \plus A| = m \big\}.$$ 
It was observed in \cite{G05} that  if $A \subset \Gamma$ is chosen uniformly at random, then the expected number of cliques of size $k$ in the Cayley sum graph of $A$ is precisely
\[ \sum_{m \ge 1} |S_k^m(\Gamma)| \, 2^{-m}.\]
Thus, by Markov's inequality, all we need do to in order to establish Theorem~\ref{maincliquetheorem} is prove that if $\eps > 0$ and if $k \geq (2 + \eps)\log_2 N$ then
\begin{equation} \label{eq:to:show} 
\sum_{m \ge 1} |S_k^m(\Z/N\Z)| \, 2^{-m} \rightarrow 0
\end{equation} 
as $N \rightarrow \infty$ through the primes. Henceforth, let us fix $\eps > 0$, set $\delta := \eps^4$, and note that, without loss of generality, we may assume that $\eps$ is arbitrarily small. To prove~\eqref{eq:to:show}, we split the sum over $m$ into four ranges:
\begin{itemize}
\item[(a)] $m \leq k/\delta$, 
\item[(b)] $k/\delta < m \leq k^{1 + \delta}$,
\item[(c)] $k^{1 + \delta} < m \leq \delta k^2$,
\item[(d)] $\delta k^2 < m \leq {k \choose 2}$.
\end{itemize}

The issue in all cases is to bound the quantity $|S_k^m(\Z/N\Z)|$. For ranges (b) and (c) we will quote results from \cite{G05}, whilst for the range (d) we proved an appropriate bound in the last section. For range (a) (small doubling) we use the main results of this paper, in particular Theorem~\ref{thm:counting}. The bound we prove is in fact closely analogous to that of Theorem~\ref{thm:counting}.

To analyse the ranges (b) and (c) we will use the following bounds from~\cite{G05}.

\begin{prop}[Proposition~23 of~\cite{G05}]\label{prop:middle:ranges}
Let $N$ be prime, and let $m,k \in \N$. There exists $r \in \N$ with $r \leq \min\{ 4m/k, k \}$ and
\begin{equation}\label{eq:r:bound} 
r \leq \frac{2m}{k} + \frac{1}{k}\binom{r}{2},
\end{equation}
such that
\[ |S_k^m(\Z/N\Z)| \leq k^{4k} N^{r}. \]
Moreover, if $m \leq k^{31/30} / 2$, then
\[ |S_k^m(\Z/N\Z)| \leq N^{r} \left( \frac{2e m}{k} \right)^k \exp\big( k^{31/32} \big). \]
\end{prop}

In fact this result is very slightly stronger than that stated in~\cite{G05}, but follows from exactly the same proof, simply by removing the final approximation $r \le 4m/k + 1$. We remark that the parameter $r$ can be set equal to the maximum possible Freiman dimension of a set $A \in S_k^m(\Z/N\Z)$ when `unwrapped', i.e., viewed as a subset of $\Z$.  

We are now ready to prove Theorem~\ref{maincliquetheorem} by showing that the sum 
\[ \sum_m |S_k^m(\Z/N\Z)| \, 2^{-m} \] tends to zero as $N \rightarrow \infty$, the sum over $m$ being taken over any one of the four ranges (a) to (d). Remember that $k \geq (2 + \eps)\log_2 N$ and that $\delta = \eps^4$.

\textsc{The range} (a). We must show that 
\[ \sum_{m \,\le\, k / \delta} |S_k^m(\Z/N\Z)| \, 2^{-m} \rightarrow 0.\] 
For consistency with earlier notation, set $K = m/k$ and note that $K \le 1/\delta$. By Lemma~\ref{lem:restricted}, every set $A \subset \Z/N\Z$ counting towards $S_k^m(\Z/N\Z)$ has $|A| = k$ and $|A + A| \leq (K + o(1))k$. Thus, by Theorem~\ref{thm:counting}, if $N$ is large enough in terms of $\eps$, then
\[ |S_k^m(\Z/N\Z)| \, \leq \, 2^{2\delta k} \binom{Kk/2}{k} N^{\lfloor K + \delta \rfloor}.\] (The change from $K$ to $K + o(1)$ has been absorbed into the $2^{2\delta k}$ term.) 
Applying the crude bounds $\binom{Kk/2}{k} \leq 2^{Kk/2} = 2^{m/2}$ and $N^{\lfloor K + \delta \rfloor} \le 2^{(1 - \eps^2)m/2}$, which holds since $k \geq (2 + \eps)\log_2 N$, this is bounded above by $2^{(1 - \eps^3)m}$. Since (trivially) $|S_k^m(\Z/N\Z)| = 0$ if $m < k - 1$, it follows that 
\[ 
\sum_{m \,\le\, k / \delta} |S_k^m(\Z/N\Z)| \, 2^{-m} \, \leq \sum_{k - 1 \,\le\, m \,\le\, k / \delta} 2^{- \eps^3 m} \, \to \, 0
\] 
as $N \rightarrow \infty$, as required. 

\textsc{The ranges} (b) \textsc{and} (c). The task is to show that 
\[ 
\sum_{k/\delta \,\le\, m \,\le\, \delta k^2} |S_k^m(\Z/N\Z)| \, 2^{-m} \rightarrow 0.
\] 
To accomplish this we will once again show that $|S_k^m(\Z/N\Z)| \le 2^{(1-\eps^3)m}$ for every $m$ in the range, which easily implies the claim. This bound follows easily from Proposition~\ref{prop:middle:ranges}, together with the observation that
$$r \, \le \, \big( 2 + \eps^2 \big) \frac{m}{k} \, \leq \, \big( 1 - \eps^2 \big) \frac{m}{\log_2 N}.$$
Indeed, since $r \le 4m/k \le \eps^2 k$, this follows immediately from~\eqref{eq:r:bound}. Hence
$$N^r \, = \, 2^{r \log_2 N} \, \le \, 2^{(1 - \eps^2)m} \qquad \textup{and} \qquad  \left( \frac{2e m}{k} \right)^k \le \, 2^{\eps^3 m},$$ 
since $k \le \delta m$, and moreover if $m \ge k^{1+\delta}$ then $k^{4k} \le  2^{\eps^3 m}$. 

By Proposition~\ref{prop:middle:ranges}, it follows that
$$|S_k^m(\Z/N\Z)| \, \le \, N^{r} \left( \frac{2e m}{k} \right)^k \exp\big( k^{31/32} \big) \, \le \, 2^{(1 - \eps^3)m},$$
if $k / \delta \le m \le k^{1+\delta}$, and
$$|S_k^m(\Z/N\Z)| \, \le \, k^{4k} N^{r} \, \le \, 2^{(1 - \eps^3)m},$$
if $k^{1+\delta} \le m \le \delta k^2$. In either case, we obtain the claimed bound. 

\textsc{The range} (d).
Here we must establish that  \[ \sum_{\delta k^2 \,\le\, m \,\le\, {k \choose 2}} |S_k^m(\Z/N\Z)| 2^{-m} \rightarrow 0.\] 
By Proposition~\ref{prop:count:toprange}, we have
$$|S_k^m(\Z/N\Z)| \, \leq \, N^{2m/k + o_{\eps}(k)} \, \le \, 2^{(1- \eps^2)m}$$ 
if $N$ is sufficiently large, since $k \geq (2 + \eps) \log_2 N$. The required bound follows.

This concludes the proof of Theorem~\ref{maincliquetheorem}.\endproof\vspace{11pt}
 
\emph{Remark.} In the analysis of (a) we used a crude bound on the binomial coefficient, but this is close to sharp when $K = 4$, corresponding to $m \approx 4k$. However sets with $m \approx 4k$ can be shown, by a more sophisticated analysis, not to make a substantial contribution to~\eqref{eq:to:show}. The main contribution is just from the endpoint $m \sim k^2/2$.

\appendix

\section{Proof of the regularity and counting lemmas}

In this appendix we shall prove the regularity and counting lemmas stated in Section~\ref{sec2-regularity}, which were the key tool in our proofs of the main theorems. \vspace{11pt}

\textsc{The regularity lemma.} Let us begin by recalling the statement of this result.

\begin{GRRL}
For every $\eps > 0$ there exists $\delta = \delta(\eps) > 0$ such that the following holds for every sufficiently large prime $p$, and every $A \subset \Z/p\Z$. There is a dilate $A^* = \lambda A$ of $A$ and a prime $q$, $\frac{1}{\eps^{10}} \le q \le p^{1-\delta}$, such that at least $(1 - \eps) q^2$ of the pairs $(A^*_i, A^*_j)$ are $\eps$-regular.
\end{GRRL}

For the definitions of the terms used here we refer the reader to Section~\ref{sec2-regularity}. The proof of the regularity lemma proceeds via an energy increment argument, similar in spirit to that of the original regularity lemma of Szemer\'edi. The basic idea is simple and quite familiar, but the details are complicated slightly by the presence of `edge effects', arising from the fact that progressions do not always subdivide neatly into subprogressions. The proof along the same lines of an analogous statement in $\F_2^n$ would avoid these technicalities (and would come with quite decent bounds); we leave this as an exercise for the interested reader.

Recall that the balanced Fourier transform $\hat{f}_A \colon \R / \Z \to \R$ of a set $A \subset \Z/p\Z$ contained in an interval $I  \subset \Z/p\Z$ is defined by
\begin{equation}\label{def:fhat}
\hat{f}_A(\theta) \, = \, \sum_{x \in I} \big( 1_A(x) - \alpha \big) e(x \theta),
 \end{equation}
where $\alpha = |A| / |I|$. 

Given a collection of disjoint sets $X_i$, the energy of $A$ relative to the $X_i$ is defined to be the square mean of the densities $|A \cap X_i|/|X_i|$. The following lemma drives the energy increment strategy.

\begin{lemma}\label{tech-lemma-4}
Let $p$ be prime, and let $I \subset \Z/p\Z$ be an interval of length $L$. Suppose that $I$ has been partitioned into disjoint progressions $P_1,\dots, P_r$ together with a ``leftover'' set $P_0$. Suppose that all the progressions $P_i$ have the same nonzero common difference $d \leq L^{1/3}$ and the same length $L'$ satisfying $L^{1/4} \ll L' \ll L^{1/4}$, and suppose that $|P_0| \ll L^{2/3}$.

Let $A \subset I$ be a set, and suppose that the balanced Fourier transform $\hat{f}_A$ satisfies $|\hat{f}_A(\theta)| \geq \delta L$ for some $\delta > 0$ and for some $\theta$ such that $\Vert \theta d \Vert \leq L^{-1/3}$. 

Then we have the energy increment
\[ \frac{1}{r} \sum_{i = 1}^r \bigg(\frac{|A \cap P_i|}{|P_i|}\bigg)^2 \, \geq \, \bigg(\frac{|A \cap I|}{|I|}\bigg)^2 + \delta^2 - O\big(L^{-1/12} \big).\]
\end{lemma}

\begin{proof} 
Note first that if $x, y \in P_i$ for some $i \in [r]$, then $|x - y| = d \ell$ for some $\ell \leq L' = L^{1/4}$. Therefore
\begin{equation}\label{eq:thetaxy}
\big| e(\theta x) - e(\theta y) \big| \, \ll \, \Vert \theta  d \Vert \cdot \ell \, \ll \, L^{-1/12}.
\end{equation}
Hence, setting $\alpha := |A \cap I|/|I|$ and $\alpha_i := |A \cap P_i|/|P_i|$ for each $i \in [r]$, we have
\begin{align*}
\delta L & \, \leq \, |\hat{f}_A(\theta)| \, = \, \bigg| \sum_{x \in I} \big( 1_A(x) - \alpha \big) e(\theta x)\bigg| \, \leq \, \sum_{i = 0}^r \Big | \sum_{x \in P_i} \big( 1_A(x) - \alpha \big) e(\theta x) \Big | \\ 
& \, \leq \, \sum_{i = 1}^r \Big | \sum_{x \in P_i} \big( 1_A(x) - \alpha \big) \Big| + O(L^{11/12}) \, = \, L' \sum_{i=1}^r \big| \alpha_i - \alpha \big| + O(L^{11/12}).
\end{align*}
Indeed, the first step is one of our assumptions, the second is the definition~\eqref{def:fhat}, the third is the triangle inequality, the fourth follows by~\eqref{eq:thetaxy} and our bound on $|P_0|$, and the final step since $|A \cap P_i| = \alpha_i |P_i| = \alpha_i L'$.  Noting that $L \geq r L'$, 
it follows immediately that
\begin{equation}\label{eq:alphadiff}
\frac{1}{r}\sum_{i = 1}^r \big| \alpha_i - \alpha \big| \, \geq \, \delta  - O\big( L^{-1/12} \big).
\end{equation}
Recall that we are aiming to give a lower bound on $\frac{1}{r} \sum_{i=1}^r \alpha_i^2$. Therefore, let us set $\overline{\alpha} := \frac{1}{r}\sum_{i=1}^r \alpha_i$ and observe that
\[\frac{1}{r}\sum_{i=1}^r \alpha_i^2 \, = \, \overline{\alpha}^2 + \, \frac{1}{r} \sum_{i = 1}^r \big( \alpha_i - \overline{\alpha} \big)^2.\]
Moreover, since $|\alpha - \overline{\alpha}| = O(L^{-1/3})$, by our bound on $|P_0|$, it follows from~\eqref{eq:alphadiff} that
\[ \frac{1}{r}\sum_{i = 1}^r \big| \alpha_i - \overline{\alpha} \big| \, \geq \, \delta  - O\big( L^{-1/12}  \big).\]
Hence, by the Cauchy-Schwarz inequality, we obtain
\[
\frac{1}{r}\sum_{i = 1}^r \alpha_i^2 \, \geq \, \overline{\alpha}^2 + \Big( \delta - O\big( L^{-1/12} \big) \Big)^2 \, \ge \, \alpha^2 + \delta^2 - O\big(L^{-1/12} \big),
\]
as required.
\end{proof}

We will use the following simple lemmas during the proof below; we gather them here for convenience.

\begin{lemma}\label{tech-lemma-3}
Let $S_1,\dots,S_k \subset [n]$ be sets, each of cardinality at most $s$. Suppose that $S_i$ intersects $S_j$ for at least $\eps k^2$ pairs $(i,j) \in [k]^2$. Then there exists a set $T \subset [k]$ of size $\eps k/s$ such that $\bigcap_{j \in T} S_j$ is non-empty. 
\end{lemma}

\begin{proof}
This is just an easy application of the pigeonhole principle. To spell it out, note that for some $i \in [k]$ we have $S_i \cap S_j \neq \emptyset$ for at least $\eps k$ values of $j \in [k]$. Since each point of intersection lies in $S_i$, the same point must be chosen at least $\eps k / s$ times.
\end{proof}

\begin{lemma}\label{tech-lemma-2}
Let $I \subset \Z$ be an interval of length $L$, suppose that $A \subset I$, and set $M = \lceil 100 L/\eps\rceil$. Then there are at most $O(\eps^{-3})$ $M$th roots of unity $\theta$ such that  $|\hat{f}_A(\theta)| \geq \eps L/2$.
\end{lemma}

\begin{proof}  
Without loss of generality, we may assume that $I = [L]$. We apply Parseval's identity in $\Z/M\Z$, which is easily established directly and states that 
\[ \sum_{r \in \Z/M\Z} \Big| \frac{1}{M} \sum_{x \in \Z/M\Z} \psi(x) e(-rx/M) \Big|^2 \, = \, \frac{1}{M} \sum_{x \in \Z/M\Z} |\psi(x)|^2
\] 
for any function $\psi \colon \Z/M\Z \rightarrow \C$. Set $\psi(x) = f_A(x) = 1_A(x) - \alpha$ if $x \in \{1,\ldots,L\}$ and $\psi(x) = 0$ otherwise. By definition, we have
\[ 
 \hat{f}_A(r / M \big) = \sum_{x \in \Z/M\Z} \psi(x) e( - rx / M),
\] 
and hence, by Parseval's identity, we obtain
\[ 
\sum_{r \in \Z/M\Z} \big| \hat{f}_A(r / M) \big|^2 \leq LM.
\] 
It follows immediately that the maximal number of $M$th roots of unity $\theta = r/M$ with $|\hat{f}_A(r/M)| \geq \eps L/2$ is at most $4M/\eps^2 L = O(\eps^{-3})$, as required.
\end{proof}

\begin{lemma}\label{tech-lemma-1}
Let $I \subset \Z$ be an interval of length $L$, let $\theta \in \R/\Z$, and suppose that $A \subset I$ is a set with $|\hat{f}_A(\theta)| \geq \eps L$. Set $M = \lceil 100 L/\eps\rceil$, and let $\tilde\theta$ be the $M$th root of unity nearest to $\theta$. Then $|\hat{f}_A(\tilde\theta)| \geq \eps L/2$.
\end{lemma}

\begin{proof} 
This follows from the mean value theorem and the fact that the derivative of $\hat{f}_A$ is bounded by $2\pi L^2$, as can be verified by term-by-term differentiation.
\end{proof}

We are now ready to prove the regularity lemma. We will define an increasing sequence of primes $q_t$  with $t = 1,2,3,\dots$ and $q_1 \sim \frac{1}{\eps^{10}}$, and a corresponding sequence of dilates $\lambda_t A$. Define the $t$th energy, $E_t$, to be
\[ 
E_t := \frac{1}{q_t}\sum_{i = 1}^{q_t} \bigg( \frac{|\lambda_t A \cap I_i(q_t)|}{|I_i(q_t)|} \bigg)^2,
\] 
and note that trivially $E_t \leq 1$. Set $L_t := p/q_t$, and note that the discrete intervals $I_i(q_t)$ have length $L_t + O(1)$. To simplify notation, let us write $A^*_i(t) = \lambda_t A \cap I_i(q_t)$ for each $i \in [q_t]$. 

Now, fix a value of $t \in \N$, and suppose that there are more than $\eps q_t^2$ pairs $(i,j)$ such that the pair $\big( A^*_i(t), A^*_j(t) \big)$ is not $\eps$-regular (the definition of this concept is given in Definition~\ref{reg-def}). For each such pair $(i,j)$ one of the following three possibilities holds\footnote{We are ignoring the very slight issue that the lengths of $I_i(q_t)$, $I_j(q_t)$ are not \emph{exactly} $L_t$, but rather $L_t + O(1)$. This is an exceptionally minor issue.}:
\begin{enumerate}
\item[(a)] $|\hat{f}_{A^*_i(t)}(\theta_{i,j})| \geq \eps L_t$ and $|\hat{f}_{A^*_j(t)}(\theta_{i,j})| \geq \eps L_t$ for some $\theta_{i,j} \in \R/\Z$;
\item[(b)] $|\hat{f}_{A^*_i(t)}(\theta_i)| \geq \eps L_t$ for some $\theta_{i} \in \R/\Z$ with $\Vert \theta_i \Vert \leq 1/\eps L_t$;
\item[(c)] $|\hat{f}_{A^*_j(t)}(\theta_j)| \geq \eps L_t$ for some $\theta_{j} \in \R/\Z$ with $\Vert \theta_j \Vert \leq 1/\eps L_t$.
\end{enumerate}
By the pigeonhole principle, one of (a), (b) and (c) holds for at least $\eps q_t^2/3$ pairs $(i,j)$. Note that, by Lemma~\ref{tech-lemma-1}, and at the cost of replacing $\eps$ by $\eps/2$, we may assume that $\theta_i$, $\theta_j$ and $\theta_{i,j}$ are all $M$th roots of unity, where $M = \lceil 100L_t/\eps\rceil$. Recall also that, by Lemma~\ref{tech-lemma-2}, if we write $\Sigma_i$ for the set of all $M$th roots of unity $\theta = r / M$ for which $\hat{f}_{A^*_i(t)}(\theta) \geq \eps L_t/2$, then we have $|\Sigma_i| = O(\eps^{-3})$.  

We claim that there exists $\theta \in \R/\Z$ and a set $\Omega \subset [q_t]$, with $|\Omega| \gg \eps^4 q_t$, such that 
\begin{equation}\label{large-fourier} 
|\hat{f}_{A^*_i(t)}(\theta)| \geq \eps L_t/2
\end{equation} 
for every $i \in \Omega$. To prove this, suppose first that (a) holds for at least $\eps q_t^2/3$ pairs $(i,j)$, and note that $\theta_{i,j} \in \Sigma_i \cap \Sigma_j$ for each such pair. By Lemma~\ref{tech-lemma-3}, it follows that there is some $\theta$ lying in $\gg \eps^4 q_t$ of the sets $\Sigma_i$, i.e., there is a set $\Omega \subset [q_t]$, with $|\Omega| \gg \eps^4 q_t$, such that~\eqref{large-fourier} holds for every $i \in \Omega$, as claimed. On the other hand, if (b) holds (say) for at least $\eps q_t^2/3$ pairs $(i,j)$, then (by the pigeonhole principle) it follows that (b) holds for at least $\eps q_t /3$ different values of $i \in [q_t]$. Since there are clearly $ \ll M / \eps L_t = O(\eps^{-2})$ different $M$th roots of unity $\theta$ for which $\Vert \theta \Vert \leq 1/\eps L_t$, another application of the pigeonhole principle gives us a $\theta \in \R / \Z$ and a set $\Omega \subset [q_t]$ of size $\ll \eps^3 q_t$ such that~\eqref{large-fourier} holds for every $i \in \Omega$, as required. Since (b) and (c) are equivalent, this proves the claim. 

Our next aim is to construct a partition of each interval $I_i(q_t)$ into arithmetic progressions, as in Lemma~\ref{tech-lemma-4}. Recall that these must all have common difference $d \leq L_t^{1/3}$ satisfying $\Vert \theta d \Vert \leq L_t^{-1/3}$ for some $\theta$ with $|\hat{f}_{A^*_i(t)}(\theta)| \geq \delta L_t$. By the claim, there exists a $\theta \in \R / \Z$ for which the last inequality holds (with $\delta = \eps / 2$) for every $i \in \Omega$. Moreover, by Dirichlet's lemma on diophantine approximation\footnote{For any $\alpha \in \R$ and any $Q \in \N$, there exists $q \in \N$ such that $q < Q$ and $\Vert q \alpha \Vert \leq 1 / Q$.} there exists $1 \leq d \leq L_t^{1/3}$ such that $\Vert \theta d \Vert \leq L_t^{-1/3}$. 

Set $\lambda_{t+1} := d^{-1} \lambda_t$, take $q_{t+1}$ to be a prime with $p^{3/4} q_t^{1/4} \ll q_{t+1} \ll p^{3/4} q_t^{1/4}$, and note that (since $L_{t+1} = p / q_{t+1}$) we have $L_t^{1/4} \ll L_{t+1} \ll L_t^{1/4}$ (as required by Lemma~\ref{tech-lemma-4}). Consider the arithmetic progressions $P_j = d I_j(q_{t+1})$ for $j \in [q_{t+1}]$, and note that these form a decomposition of $\Z / p\Z$. Note that at most $d q_t$ of these progressions are not contained in some interval $I_i(q_t)$ (since each must contain an element within distance $d/2$ of the endpoint of some interval). Call such progressions bad, and the remaining progressions good.

Set $R_i = \{ j : d I_j(q_{t+1}) \subset I_i(q_t) \}$ and $r_i = |R_i|$. We claim that, by Lemma~\ref{tech-lemma-4}, we have
\[ 
\frac{1}{r_i} \sum_{j \in R_i} \bigg( \frac{|\lambda_t A \cap P_j|}{L_{t+1}} \bigg)^2 \geq  \bigg( \frac{|A^*_i(t)|}{L_t} \bigg)^2 + \frac{\eps^2}{4} 1_{i \in \Omega} - O\big( L_t^{-1/12} \big).
\]
Indeed, if $i \in \Omega$ then this follows by the lemma, noting that $P_0 = I_i(q_t) \setminus \bigcup_j P_j$ contains at most $2d L_{t+1} \ll L_t^{2/3}$ elements. On the other hand, if $i \not\in \Omega$ then it follows by the Cauchy-Schwarz inequality.

Noting that $|\lambda_t A \cap P_j| = |\lambda_t A \cap d I_j(q_{t+1})| = |\lambda_{t+1} A \cap I_j(q_{t+1})| = |A^*_j(t+1)|$, and that
\[ 
r_i = L_t/L_{t+1} + O(d) = q_{t+1}/q_t +O(d),
\] 
this implies that
\[ 
\frac{1}{q_{t+1}} \sum_{j \in R_i} \bigg( \frac{|A^*_j(t+1)|}{L_{t+1}} \bigg)^2  \geq  \frac{1}{q_t} \bigg( \frac{|A^*_i(t)|}{L_t} \bigg)^2 + \frac{\eps^2}{4 q_t} 1_{i \in \Omega} - O\bigg( \frac{d}{q_{t+1}} + \frac{1}{q_t L_t^{1/12}} \bigg).
\]
Summing over $i \in [q_{t+1}]$, and recalling that $|\Omega | \gg \eps^4 q_t$, we obtain
\begin{align*}
E_{t+1} & \, = \, \frac{1}{q_{t+1}} \sum_{j = 1}^{q_{t+1}} \bigg( \frac{|A^*_j(t+1)|}{L_{t+1}} \bigg)^2 \, \ge \, \frac{1}{q_{t+1}} \sum_{i = 1}^{q_t} \sum_{j \in R_i} \bigg( \frac{|A^*_j(t+1)|}{L_{t+1}} \bigg)^2  \\
& \, \geq \, \frac{1}{q_t} \sum_{i = 1}^{q_t} \bigg( \frac{|A^*_i(t)|}{L_t} \bigg)^2 + \eps^6 c - O\bigg( \frac{d q_t }{q_{t+1}} + \frac{1}{L_t^{1/12}} \bigg) \, \ge \, E_t + \frac{\eps^6 c}{2}
\end{align*}
for some absolute $c > 0$, provided that $L_t \gg \eps^{-72}$, since $d q_t \le L_t^{1/3} q_t \ll L_t^{-1/12} q_{t+1}$.


Since the energy $E_t$ is always bounded by $1$, the iteration stops at time $t \ll \eps^{-6}$, provided that $L_t \gg \eps^{-72}$ at that point. Since $L_t \sim (\eps^{10} p)^{(1/3)^t}$, this will be so if $\eps \gg (\log \log p)^{-1/6}$. This completes the proof of the regularity lemma.

\vspace{11pt}

\textsc{Proof of the counting lemma.} Let us begin by recalling the statement.

\begin{Count}
Let $\eps, L$ be positive parameters with $L > 16/\eps$. Suppose that $I, I' \subset \Z/p\Z$ are intervals with $|I|, |I'| = L+ O(1)$. Suppose also that the pair of sets $A \subset I$ and $A' \subset I'$ is $\eps^7$-regular and that $|A|, |A'|\geq \eps L$. Then $|A + A'| \geq (2 - 8\eps) L$.
\end{Count}

\begin{proof}
We will ignore, for notational simplicity, the fact that $I, I'$ do not have length \emph{exactly} $L$. This makes no material difference. Note also that we may assume that $\eps < 1/2$, since otherwise the claim holds trivially. Suppose the result is false. Then there is a set $S \subset I + I'$, $|S| = 8 \eps L$, such that $(A + A') \cap S = \emptyset$. This we may write as
\[ \sum_{x \in I} \sum_{x' \in I'} 1_A(x) 1_{A'}(x') 1_S(x + x') = 0.\] 
Write $\alpha := |A|/|I|$ and $\alpha' := |A'|/|I'|$. By assumption, $\alpha, \alpha' \geq \eps$. Writing $1_A = \alpha 1_I + f_A$ and $1_{A'} = \alpha' 1_{I'} + f_{A'}$, we may expand as a sum of four terms. There is a ``main term'' 
\[ M := \alpha\alpha' \sum_{x \in I} \sum_{x' \in I'}  1_S(x + x') \] and three further ``error'' terms $E_1, E_2$ and $E_3$. 
We begin by giving a lower bound for the main term $M$. There are only two points of $I + I'$ with a unique representation as $x + x'$ (the two endpoints), two points with just $2$ representations, and so on. Therefore we have
\[ \sum_{x \in I} \sum_{x' \in I'} 1_S(x + x')  \geq 1 + 1 + 2 + 2 + 3 + \dots + \lceil \textstyle\frac{1}{2}|S| \rceil .\] 
This is at least $\frac{1}{2} + 1 + \frac{3}{2} + 2 + \frac{5}{2} + \dots + \frac{1}{2}|S|$. Note that, since $L > 16/\eps$, we have $|S| > 2$. Therefore this is at least $\frac{1}{8}|S|^2\geq 5 \eps^2 L^2$, and hence $M \geq 5 \eps^4 L^2$.

We turn now to the error terms $E_1, E_2, E_3$. We will show that $E_1 + E_2 + E_3 > - 4 \eps^4 L^2$, which combined with the observations above gives a contradiction. We begin by expressing the error terms using the Fourier transform. One easily checks\footnote{Here $\hat{1}_S(\theta) = \sum_{x \in S} e(x \theta)$ and $\hat{f}_A(\theta)$ is as defined in Section~\ref{sec2-regularity}. Since $\int_0^1 e(x\theta) \dt = 1_{\{x = 0\}}$, the three identities all follow easily from the definitions.} that:
\[ E_1 = \alpha \int^1_0  \hat{1}_I(\theta) \hat{f}_{A'}(\theta) \hat{1}_S(-\theta) \dt,\]
\[ E_2 = \alpha' \int^1_0  \hat{1}_{I'}(\theta) \hat{f}_A(\theta) \hat{1}_S(-\theta) \dt\] and 
\[ E_3 = \int^1_0 \hat{f}_A(\theta) \hat{f}_{A'}(\theta) \hat{1}_S(-\theta) \dt.\]
Let us bound these in turn. We claim first that $|E_1| \leq 8 \eps^{15/2} L^2 < \eps^4 L^2$ (recall that $\eps < 1/2$); the same bound holds for $|E_2|$ analogously. Let $U = \{ \theta \in [0,1] : \Vert \theta \Vert \leq 1/\eps^7 L \}$; we shall split into two parts, depending on whether or not $\theta \in U$. Since the pair $(A, A')$ is $\eps^7$-regular, it follows that $|\hat{f}_{A'}(\theta)| \leq \eps^7 L$ on $U$, and therefore
\begin{align*}
& \bigg| \int_{\theta \in U} \hat{1}_I(\theta) \hat{f}_{A'}(\theta) \hat{1}_S(-\theta) \dt \bigg| \, \leq \, \eps^7 L \int^1_0 |\hat{1}_I(\theta)||\hat{1}_S(-\theta)| \dt \\ 
& \hspace{2.5cm} \, \leq \, \eps^7 L \bigg( \int^1_0 |\hat{1}_I(\theta)|^2 \dt \bigg)^{1/2} \bigg( \int^1_0 |\hat{1}_S(\theta)|^2 \dt \bigg)^{1/2} \leq \, 3\eps^{15/2} L^2.
\end{align*}
The last step follows since, by Parseval's identity, 
\[ 
\int^1_0 |\hat{1}_S(\theta)|^2\, \dt = \sum_{x \in \Z} 1_S(x)^2  = 6\eps L \qquad \mbox{and} \qquad \int^1_0 |\hat{1}_I(\theta)|^2 \, \dt = L.
\]
On the range $\Vert \theta \Vert > 1/\eps^7L$ we use the well-known bound
\[ 
|\hat{1}_I(\theta)| \leq \frac{2}{\Vert \theta \Vert} \leq 2\eps^7 L,
\] 
which follows by explicitly computing $\hat{1}_I(\theta)$ by summing a geometric series. As a result of this, we obtain
\begin{align*}
& \bigg| \int_{\theta \not\in U} \hat{1}_I(\theta) \hat{f}_{A'}(\theta) \hat{1}_S(-\theta) \dt \bigg| \, \leq \, 2\eps^7 L \int^1_0 |\hat{f}_{A'}(\theta)||\hat{1}_S(-\theta)| \dt \\ 
& \hspace{2.5cm} \, \leq \,  2\eps^7 L \bigg( \int^1_0 |\hat{f}_{A'}(\theta)|^2 \dt \bigg)^{1/2} \bigg( \int^1_0 |\hat{1}_S(\theta)|^2 \dt \bigg)^{1/2}  \leq \, 5 \eps^{15/2} L^2.
\end{align*}
The final inequality again follows by Parseval's identity, since 
\[ 
\int^1_0 |\hat{f}_{A'}(\theta)|^2\, \dt \, = \, \sum_{x \in \Z} f_{A'}(x)^2 \, \le \, L.
\]
Putting these two bounds together, we obtain $|E_1| \leq 8 \eps^{15/2} L^2$, as claimed. 

Turning now to $E_3$, recall that since the pair $(A, A')$ is $\eps^7$-regular, we have $|\hat{f}_A(\theta)| |\hat{f}_{A'}(\theta)| \le \eps^7 L^2$ for every $\theta \in \R/\Z$. It follows that
\begin{align*}
|E_3| & \, \leq \, \sup_{\theta} \Big( |\hat{f}_A(\theta)|^{1/2}|\hat{f}_{A'}(\theta)|^{1/2} \Big) \int^1_0 |\hat{f}_A(\theta)|^{1/2}|\hat{f}_{A'}(\theta)|^{1/2} |\hat{1}_S(-\theta)| \dt \\ 
& \, \leq \, \eps^{7/2} L \bigg( \int^1_0 |\hat{f}_A(\theta)|^2 \dt \bigg)^{1/4} \bigg( \int^1_0 |\hat{f}_{A'}(\theta)|^2 \dt\bigg)^{1/4} \bigg( \int^1_0 |\hat{1}_S(\theta) |^2 \dt \bigg)^{1/2} \, \leq \, 3\eps^4 L^2,
\end{align*}
where the final inequality again follows by Parseval's identity. 

Putting everything together, we obtain $E_1 + E_2 + E_3 > - 5 \eps^4 L^2$ and $M \geq 5 \eps^4 L^2$, and hence 
\[ M + E_1 + E_2 + E_3 > 0,\] 
which contradicts our choice of $S$. This completes the proof of the counting lemma.
\end{proof}


\begin{thebibliography}{99}

\bibitem{ABMS} N.~Alon, J.~Balogh, R.~Morris and W.~Samotij, \emph{A refinement of the Cameron-Erd\H{o}s Conjecture}, to appear in Proc. London Math. Soc.

\bibitem{blr} Y.F.~Bilu, V.F.~Lev and I.Z.~Ruzsa, \emph{Rectification principles in additive number theory}, Dedicated to the memory of Paul Erd\H{o}s,  
Discrete Comput. Geom. \textbf{19} (1998), 343--353.

\bibitem{BL} B.~Bollob\'as and I.~Leader, \emph{Compressions and isoperimetric inequalities}, J. Combin. Theory Ser. A \textbf{56} (1991), 47--62.

\bibitem{bourgain-multilinear} J.~Bourgain, \emph{Multilinear exponential sums in prime fields under optimal entropy condition on the sources}, Geom. Funct. Anal. \textbf{18} (2009), 1477--1502. 

\bibitem{Chang} M.-C. Chang, \emph{A polynomial bound in Freiman's theorem}, Duke Math. J. \textbf{113} (2002), 399--419.


\bibitem{Frei59} G.~A.~Freiman, \emph{The addition of finite sets. I} (Russian), Izv. Vys\v{s}. U\v{c}ebn. Zaved. Matematika \textbf{13} (1959), 202--213.

\bibitem{garaev-quant} M.~Garaev, \emph{A quantified version of Bourgain's sum-product esimate in $\F_p$ for subsets of incomparable sizes}, Electronic J. Combinatorics \textbf{15} (2008). 

\bibitem{graham-ringrose} S.W.~Graham and C.J.~Ringrose, \emph{Lower bounds for least quadratic non-residues}, in Analytic Number Theory (Allerton Park 1989), Progress in Mathematics \textbf{85},  pp269--309, Birkh\"auser (Basel) 1990.

\bibitem{GEdin} B.~J.~Green, \emph{Edinburgh lecture notes on Freiman's theorem}, preprint. \\ available at \texttt{https://www.dpmms.cam.ac.uk/$\sim$bjg23/papers/convexnotes.pdf}.

\bibitem{G05} B.~J.~Green, \emph{Counting sets with small sumset, and the clique number of random Cayley graphs}, Combinatorica \textbf{25} (2005), 307--326.

\bibitem{green-arithreg} B.~J.~Green, \emph{A Szemer\'edi-type regularity lemma in abelian groups}, Geom. Funct. Anal. \textbf{15} (2005), 340--376.

\bibitem{GRprime} B.~J.~Green and I.~Z.~Ruzsa, \emph{Counting sumsets and sum-free sets modulo a prime}, Studia Sci. Math. Hungarica \textbf{41} (2004), 285--293.

\bibitem{GRrect} B.~J.~Green and I.~Z.~Ruzsa, \emph{Sets with small sumset and rectification}, Bull. London Math. Soc. \textbf{38} (2006), 43--52.
  
\bibitem{gt-arithreg} B.~J.~Green and T.~C.~Tao, \emph{An arithmetic regularity lemma, associated counting lemma, and applications}, in \emph{An irregular mind - Szemer\'edi is 70}, Bolyai Soc. Math. Stud. \textbf{21}, 261--334, J\'anos Bolyai Math. Soc., Budapest, 2010.
 
\bibitem{KSVW} M.~Krivelevich, B.~Sudakov, V.~H.~Vu and N.~C.~Wormald, \emph{Random regular graphs of high degree}, Random Structures Algorithms \textbf{18} (2001), 346--363.

\bibitem{Pollard} J.~M.~Pollard, \emph{A generalization of the theorem of Cauchy and Davenport}, J. London Math. Soc. \textbf{2} (1974), 460--462.

\bibitem{ruzsa-model} I.~Z.~Ruzsa, \emph{Arithmetical progressions and the number of sums}, Period. Math. Hungar. \textbf{25} (1992), 105--111. 

\bibitem{sanders-survey} T.~Sanders, \emph{The structure theory of set addition revisited}, Bull. Amer. Math. Soc. \textbf{50} (2013), 93--127.

\bibitem{Schoen} T.~Schoen, \emph{The cardinality of restricted sumsets}, J. Number Theory \textbf{96} (2002), 48--54.

\bibitem{TV} T.~C.~Tao and V.~H.~Vu, \emph{Additive combinatorics}, Cambridge University Press, 2006.

\bibitem{WW} D.-L~Wang and P.~Wang, \emph{Discrete isoperimetric problems}, SIAM J. Appl. Math. \textbf{32} (1977), 860--870.

\end{thebibliography}
\end{document}